\documentclass[12pt]{article}
\usepackage{latexsym,amssymb, amsthm, epsfig}
\usepackage{color}
\usepackage{graphicx}
\usepackage{amsmath}
\usepackage{amsfonts}

\usepackage{a4wide}
\newtheorem{theorem}{\bf Theorem}[section]

\newtheorem{proposition}{\bf Proposition}[section]
\newtheorem{remark}{\bf Remark}[section]

\newcounter{for}[section]
\newcommand{\be}[1]{\addtocounter{for}{1} \begin{equation}\label{#1}}
\newcommand{\ee}{\end{equation}}

\def\E{{\mathbb E}}
\def\Exi{{\E^N_\xi}}

\def\P{{\mathbb P}}
\def\R{{\mathbb R}}

\def\cL{{\mathcal{L}}}

\def\M{{\cal{M}\rm}}
\def\qsd{{\sc{qsd}}}
\def\fv{{\sc{fv }}}

\def\({{\Bigl(}}
\def\){{\Bigr)}}
\def\reff#1{(\ref{#1})}
\def\one{{\mathbf 1}}
\def\ind{{\mathbf 1}}
\def\square{\ifmmode\sqr\else{$\sqr$}\fi}
\def\sqr{\vcenter{
         \hrule height.1mm
         \hbox{\vrule width.1mm height2.2mm\kern2.18mm\vrule width.1mm}
         \hrule height.1mm}}                  

\def\o{\omega}
\def\v{\varphi}
\newcommand {\cro}[1] {\left[ {#1} \right]}
\newcommand {\acc}[1] {\left\{ {#1} \right\}}
\newcommand {\pare}[1] {\left( {#1} \right)}

\theoremstyle{plain}
\newtheorem{lema}{Lemma}[section]

\theoremstyle{definition}

\theoremstyle{remark}

\begin{document}


\title{Quasi-stationary distributions and Fleming-Viot processes in finite
  spaces} \author{Amine Asselah\\
  Universit\'e Paris-Est,
  UMR CNRS 8050\\
\\
Pablo A. Ferrari\\
U. de S\~ao Paulo, U. de Buenos Aires and CONICET\\
\\
Pablo Groisman\\
Universidad de Buenos Aires and CONICET\\
}


\date{}
\maketitle
\begin{abstract} 
  Consider a continuous time Markov chain with rates $Q$ in the state space
  $\Lambda\cup\{0\}$. In the associated
  Fleming-Viot process $N$ particles evolve independently in $\Lambda$ with
  rates $Q$ until one of them attempts to jump to state $0$. At
  this moment the particle comes back to $\Lambda$ instantaneously, by jumping
  to one of the positions of the other particles, chosen uniformly at random.
  When $\Lambda$ is finite, we show that the empirical distribution of the
  particles at a fixed time converges as $N\to\infty$ to the distribution of a
  single particle at the same time conditioned on not touching $\{0\}$. 
Furthermore,
  the empirical profile of the unique invariant measure for the Fleming-Viot
  process with $N$ particles converges as $N\to\infty$ to the unique
  quasi-stationary distribution of the one-particle motion. A key element of the
  approach is to show that the two-particle correlations is of order $1/N$.
\end{abstract}

{\em Keywords and phrases}:
Quasi-stationary distributions. Fleming-Viot process.

{\em AMS 2000 subject classification numbers}: 60K35, 60J25.

{\em Running head}: Fleming-Viot process and \qsd.

\section{Introduction}


Let $\Lambda$ be a finite or countable state space, whose elements are also
called \emph{sites}. Let $Q=(q(x,y);\,{x,y\in\Lambda\cup\{0\}})$ be the
transition rates matrix of an irreducible continuous time Markov jump process on
$\Lambda\cup\{0\}$.  

For each integer $N>1$, the Fleming-Viot ({\sc fv}) process with $N$ particles,
is a continuous time Markov process, $\xi_t\in \Lambda^N$, $t\ge 0$, where
$\xi_t(i)$ denotes the position of the particle $i$ at time $t$. The generator
$\cL^N$ acts on functions $f:\Lambda^{N}\to\R$ as follows
\begin{equation}\label{generator}
  \cL^N f(\xi) = \sum_{i=1}^N \sum_{x\in\Lambda\setminus\{\xi(i)\}}
  \Bigl[q(\xi(i),x) + q(\xi(i),0)\, 
\frac{\sum_{j\not=i}^N \one_{\{\xi(j)=x\}}}{N-1}\Bigr] (f(\xi^{i,x}) - f(\xi)),
\end{equation}
where $\xi^{i,x}(i)=x$, and for $j\not=i$, $\xi^{i,x}(j) = \xi(j)$. We set
$\E^N_{\xi}[f(\xi_t)]=\exp(t\cL^N)f(\xi)$. In words, each particle moves
independently of the others as a continuous time Markov process with rates $Q$,
but when it attemps to jump to state 0, it comes back immediately to
$\Lambda$ by jumping to the position of one of the other particles chosen
uniformly at random.

This type of \fv process was introduced by Burdzy, Holyst and
 March in \cite{burdzy} for
Brownian motions on a bounded domain. In their model $N$ brownian particles
evolve independently until one of them reaches the boundary, which plays the
role of state 0. In \cite{burdzy}, the authors prove that
the empirical profile of the invariant measure
converges in this case to the first eigenfunction of the Laplacian on the domain
with homogeneous Dirichlet boundary conditions.

Denote by $\eta(\xi,x)$ the number of $\xi$ particles at site $x$, and
by $m(\xi)$ the  
empirical measure induced by a configuration $\xi\in \cup_N \Lambda^N$ 
\begin{equation}\label{a31}
\eta(\xi,x) := \sum_{i=1}^N \one\{\xi(i)=x\}\quad\text{and}\quad
\displaystyle m(\xi) := \frac{\sum_{x\in \Lambda} \eta (\xi,x)\delta_x}
{\sum_{x\in \Lambda}\eta(\xi,x)}
.
\end{equation}
We also use $m_x(\xi)$ to denote $m(\xi)(x)$ and
$q(x,x)=-\sum_{y\in\Lambda\cup\{0\}\setminus\{x\}} q(x,y)$. With this notation,
the time-derivative of $\E^N_{\xi}[m_x(\xi_t)]$ is easily seen to be
\begin{equation}\label{generator.eta}
 \frac{d \E^N_{\xi}[m_x(\xi_t)]}{dt}=
\sum_{y\in \Lambda}
q(y,x)\Exi[m_y(\xi_t)]
  +\frac{N}{N-1}\sum_{y\in \Lambda} q(y,0)\Exi[m_y(\xi_t)m_x(\xi_t)]\,.
\end{equation}

Consider the process on $\Lambda\cup\{0\}$ generated by $Q$, with initial
law $\mu$ and denote $T_t\mu$ its law at time $t$ conditioned on not having
touched $\{0\}$ up to time $t$. In other words, for all $x\in \Lambda$
\begin{equation}
T_t\mu(x) = \frac{\sum_{y\in\Lambda} \mu(y)
\exp(tQ)(y,x)}{1-\sum_{y\in\Lambda} \mu(y) \exp(tQ)(y,0)}.
\end{equation}
Let $\M$ be the space of probability measures on $\Lambda$. 
Then $\{T_t,t\ge 0\}$ is
a semi-group on $\M$ and $T_t\mu$ is the unique solution to the Kolmogorov
forward equations: $T_0(\mu)(x)=\mu(x)$, and
\begin{equation}
 \label{kolmogorov.cond.evolution}
 \frac{d}{dt}T_t\mu(x) = \sum_{y\in \Lambda}
 q(y,x)\,T_t\mu(y)+
 \sum_{y\in \Lambda} q(y,0)\,T_t\mu(y)\,T_t\mu(x).
\end{equation}

A \emph{quasi-stationary distribution} (\qsd) for $Q$ is a probability
measure $\nu$ on $\Lambda$ that is invariant under $\{T_t,t\ge 0\}$, that is
$$
T_t\nu= \nu, \quad \text{for all} \quad t\ge 0.
$$

Now, assume that $\mu$ is close to $m(\xi)$, and look at \reff{generator.eta}
and \reff{kolmogorov.cond.evolution}. A natural approach to show that $T_t\mu$
is close to $\E_{\xi}^N [m(\xi_t)]$ is by
establishing that the occupation numbers of two distinct sites, at time $t$,
become independent when $N$ tends to infinity (the so-called {\it propagation of
  chaos}).  For this purpose, Ferrari and Maric \cite{fm}, estimate the
correlation of two $\xi$-particles, when $\Lambda$ is only assumed countable.
\begin{proposition}[Ferrari and Maric \cite{fm}, Proposition
  3.1]\label{chaos-fm}
Let $\mu$ be any probability measure on $\Lambda$, and $\mu^{\otimes N}$ the
product probability on $\Lambda^N$. Then, there is a constant
$\kappa$, such that for any $x,y\in \Lambda$, and $t>0$,
\begin{equation}\label{ferrari-nevena0}
  \Bigl|\int \Exi[m_x(\xi_t)m_y(\xi_t)]d\mu^{\otimes N}(\xi)
  -\int \Exi[m_x(\xi_t)]d\mu^{\otimes N}(\xi)\int\Exi[m_y(\xi_t)]
  d\mu^{\otimes N}(\xi)\Bigr|\le \frac{e^{\kappa t}}{N}.
\end{equation}
\end{proposition}
For countable $\Lambda$, the ergodicity of the \fv process is not guaranteed,
neither is the convergence of the invariant measures as $N\to \infty$.  Under
further assumptions (see \reff{def-C}), there exists a unique invariant measure for the \fv process
and its profile converges as $N\to\infty$ to the unique \qsd.

\begin{theorem}[Ferrari and Maric \cite{fm}, Theorem 1.2 and 1.4]\label{fm-theo}
Let $\mu$ be any probability measure on $\Lambda$, and $\mu^{\otimes N}$ the
product probability on $\Lambda^N$. Then, for each $x\in \Lambda$
\begin{equation}\label{ferrari-nevena1}
\lim_{N\to \infty} \int_{\Lambda^N} 
\Exi\cro{\pare{m_x(\xi_t)-T_t\mu(x)}^2}\,d\mu^{\otimes N}(\xi)=0.
\end{equation}
Moreover, if 
\begin{equation} 
\label{def-C} 
\sum_{z\in\Lambda} \inf_{x\in\Lambda\setminus
    \{z\}}q(x,z)\;> \;\max_{x\in \Lambda}q(x,0),
\end{equation}
then there exists a measure $\nu$ on $\Lambda$ and for each $N$ there exists a
unique stationary measure for the Fleming-Viot process, say $\lambda^N$ such
that for each $x\in \Lambda$
\begin{equation}\label{ferrari-nevena2}
\lim_{N\to \infty} \int_{\Lambda^N} \pare{
m_x(\xi) - \nu(x)}^2\, d\lambda^{N}(\xi)=0.
\end{equation}
Furthermore, $\nu$ is the unique quasi-stationary distribution of $Q$.
\end{theorem}

Here, we have similar results. We prove the asymptotic independence 
of two particles
but, in contrast to \cite{fm}, we consider deterministic initial
configurations and obtain bounds on the correlations that hold uniformly on the
initial distribution of the particles. This result also holds for
countable~$\Lambda$.

\begin{proposition}\label{chaos-pablo}
For each $t>0$, and any $x,y\in\Lambda$ 
\begin{equation}\label{correlations}
\sup_{\xi\in\Lambda^N}\big|\Exi[m_x(\xi_t)m_y(\xi_t)]-
\Exi[m_x(\xi_t)]\,\Exi[m_y(\xi_t)]\big|\le \frac{\ind_{\{x=y\}}}{N}+
\frac{2}{N}\pare{e^{2Ct}-1}.
\end{equation}
\end{proposition}

Our proof yields also that any finite number of particles evolve independently
in the limit $N\to\infty$. A similar result is obtained by Grigorescu and Kang
\cite{gk2}, following the approach \cite{gk1} by the same authors.

\paragraph{Finite state spaces.} In the rest of the paper we consider a finite
$\Lambda$. In this case, for each $N\ge 2$ the {\fv} process is an irreducible
pure-jump Markov process on the finite state space $\Lambda^N$. Hence it is
ergodic, that is, there exists a unique stationary measure for the process and
starting from any measure, the process converges to the stationary measure. We
still denote this measure with $\lambda^N$.

When $\Lambda$ is finite, Darroch and Seneta \cite{darroch} prove that $T_t\mu$
converges exponentially fast to a probability measure $\nu$, uniformly in the
initial measure. The measure $\nu$ is the unique quasi-stationary distribution
of $Q$.
\begin{theorem}[Darroch and Seneta, 1967]\label{ds}
  Assume $\Lambda$ is finite, and that the process on $\Lambda$ with rates
  $\{q(x,y), x,y\in \Lambda\}$ is irreducible. Then there exists $\theta>0$,
  such that
\begin{align}\label{da-se}
\sup_{\mu \in \M} \|T_t\mu-\nu\| \le e^{-\theta t}.
\end{align}
\end{theorem}
We have used the total variation norm
$\|\mu-\nu\|=\sum_{x\in \Lambda}|\mu(x)-\nu(x)|$.

The asymptotic independence of Proposition \ref{chaos-pablo} implies naturally
the convergence of the empirical means in the \fv process to the conditioned
distribution $T_tm(\xi)$, uniformly in $\xi$. 
Moreover, $\Lambda$ finite yields that $T_t(m(\xi))$ is close to the unique \qsd,
uniformly in $\xi$ as implied by \reff{da-se}. 
These two facts imply the following result.

\begin{theorem} \label{main.theorem} Assume $\Lambda$ finite. For any positive
  time $t$
\begin{equation}\label{key.1}
\lim_{N\to\infty} \sup_{\xi \in \Lambda^N} 
\Exi[\|m(\xi_t)-T_tm(\xi)\|]=0.
\end{equation}
Moreover,
\begin{equation}\label{key.2}
\lim_{N\to\infty} 
\int_{\Lambda^N}\|m(\xi)- \nu\|\, d\lambda^N(\xi) = 0.
\end{equation}
\end{theorem}
\begin{remark}\label{rem-corr}
Note that the limit \eqref{key.2} follows 
from \eqref{key.1} using \eqref{da-se}:
\be{triangle0}
\begin{split}
\int_{\Lambda^N} & \|m(\xi)- \nu\|\,d\lambda^N(\xi)  
=\int_{\Lambda^N} \Exi[\|m(\xi_t) - \nu\|]\,d\lambda^N(\xi)\\
&\le \int_{\Lambda^N} \Exi[\|m(\xi_t)-T_tm(\xi)\|]\,d\lambda^N(\xi) 
+\int_{\Lambda^N} \|T_tm(\xi) -\nu\|\,d\lambda^N(\xi)\\
&\le \sup_{\xi\in \Lambda^N} \Exi[\|m(\xi_t)-T_tm(\xi)\|]+
\sup_{\mu\in \M}  \|T_t\mu -\nu\|.
\end{split} 
\ee 
On the other hand, \eqref{key.2} implies readily
(using that $0\le m_x\le 1$ for all $x\in \Lambda$), 
that for any subset $U\subset \Lambda$
\be{eq-patu}
\lim_{N\to\infty} \int \prod_{x\in U} m_x \ d\lambda^N\ =
\ \prod_{x\in U} \nu(x).
\ee
We omit the obvious inductive proof.
\end{remark}

The paper is organized as follows.  In Section \ref{graphical.construcion}, we
construct the process \emph{\`a la Harris} following \cite{fm}.  We use this
construction to estimate the correlations and prove
Proposition~\ref{chaos-pablo} in Section \ref{sec.correlations}.  Finally, in
Section \ref{sec.conv.means}, we prove Theorem~\ref{main.theorem}.

\section{Graphical construction}
\label{graphical.construcion}
The graphical construction is used to prove the asymptotic 
independence property. A realization of the process $(\xi_t,\,t\ge
0)$ is a deterministic function of a realization of a marked point process. All
initial conditions $\xi$ use the same realization of the marked point process.

Let $C:=\max_{x\in \Lambda}q(x,0)$ be the (maximum) 
absorption rate, and 
\[
\bar q := \sup_{x\in\Lambda} \sum_{y\in\Lambda\setminus \{x\}} q(x,y);
\quad p(x,y):=\frac{q(x,y)}{\bar q},\; y\neq x;
\quad p(x,x):=1-\sum_{y\in\Lambda\setminus \{x\}} p(x,y).
\]
To each particle $i$, we associate two independent Poisson marked processes
$(\o_i^I,\o^V_i)$, which we call respectively the {\it internal} and {\it voter}
point processes, described as follows.
\begin{itemize}
\item The internal process is defined on 
$\R\times \Lambda^\Lambda$ with intensity measure $\bar q dt\ d\gamma(F)$ with
\[
d\gamma(F)=\prod_{x\in \Lambda} p(x,F(x)),\quad\forall F
\in \Lambda^\Lambda.
\]
that is, $\gamma$ is the joint distribution of independent random variables with
marginal distributions 
$\{p(x,\cdot),\ x\in\Lambda\}$, so that to each state $x$ the
random mark $F$ assigns a state $y=F(x)$ with probability $p(x,y)$. An internal
marked-time is $(t,F)$, which means that if at at time $t-$, particle $i$ is
at site $x$, then at time $t$, particle $i$ jumps to site $F(x)$. This gives the
correct rate $q(x,y)=\bar q\, p(x,y)$ for jumps from $x$ to $y$.

\item The voter point process is defined on $\R\times(\{1,\dots,N\}\setminus\{i\})\times
  \{0,1\}^\Lambda$ with measure intensity $C\, dt\, d\beta_i(j)\,d\gamma'(\zeta)$,
  where $\beta_i$ is the uniform probability on $\{1,\dots,N\}\setminus \{i\}$,
  and
  $\gamma'$ is the joint distribution of independent Bernoulli random variables
  with parameters $\frac{q(x,0)}{C}$, $x\in\Lambda$.  A voter marked-time is
  $(t,j,\zeta)$, where $j$ corresponds to a reincarnation label, and $\zeta$
  takes into account the position-dependent rate: if the $i$-th particle is at
  position $x$ at time $t-$, it jumps to the position of particle $j$ at time
  $t$ only if $\zeta(x)=1$, yielding the correct rate $q(x,0)/(N-1)$.
\end{itemize}
We call $\o=((\o_i^I,\o_i^V),i\in \{1,\dots,N\})$ an i.i.d.\,sequence
of stationary marked point processes associated with labeled particles.
Finally, for any subset of labels $a\subset \{1,\dots,N\}$, we denote
by $\o_a$ the processes associated with labels $a$. For any
$s<t$ reals, we denote by $\o_a[s,t[$ and $\o_a[s,t]$ the projections of the marked-times
in the time period $[s,t[$ and $[s,t]$, respectively. 

We construct $\{\xi_t,t\ge 0\}$ in such a way that $\xi_t$ is a function of the
initial configuration $\xi_0$ and the time-marks $\o[0,t]$, $t\ge 0$.  Fix an
initial configuration $\xi_0\in \Lambda^N$, and $t>0$.  There is, almost surely,
a finite number of time-marks within $[0,t]$, say $K$, and let $\{b_k,k\le K\}$
be the ordered time-realizations. We build $\xi_t$ inductively as follows.
\begin{itemize}
\item At time $0$, the configuration is $\xi_0$.
\item Assume $\xi_{b_k}$ is known. For $t\in [b_k,b_{k+1}[$, set $\xi_t=
  \xi_{b_k}$. We describe now $\xi_{b_{k+1}}$.
\begin{itemize}
\item If $b_{k+1}$ corresponds to an internal time of particle $i$ and mark $F$,
  we move particle $i$ to $F(x)$ where $x=\xi_{b_k}(i)$.  This 
move occurs with rate $\bar q\times \frac{q(x,F(x))}{\bar q}
= q(x,F(x))$.
\item If $b_{k+1}$ corresponds to a voter time of particle $i$ and mark
  $(j,\zeta)$, we move particle $i$ to the position of particle $j$ if
  $\zeta(x)=1$ where $x=\xi_{b_k}(i)$.  This move occurs with rate $C\times
  \frac{q(x,0)}{C}\times\frac{1}{N-1}= q(x,0)/(N-1)$.
\end{itemize}
\end{itemize}
It is easy to check that $\{\xi_t,t\ge 0\}$, as constructed above, has generator
given by \eqref{generator}.

By translation invariance of the law of $\omega$, if we use the marks
$\omega[-t,0]$ instead of the marks $\omega[0,t]$, the configuration so obtained
has the same law of $\xi_t$ as constructed above. We abuse notation and call
$\xi_t$ the configuration constructed with the marks $\omega[-t,0]$.
For each particle label $i$, we build simultaneously a set of labels $\psi^i(t)$
of particles which could potentially influence $\xi_t(i)$; this set is also a
function of $\o[-t,0]$.  First, the process $t\mapsto\psi^i(t)$ may only change
at the time-realizations of the voter process $\o^V$, and it changes as
follows. Let $-v$ be the largest time realization of $\o_i^V[-t,0[$, and let
$(j,\zeta)$ its associated mark. Then, for $0\le s<v$, we set $\psi^i(s)=\{i\}$,
and $\psi^i(v)=\{i,j\}$ (regardless the values of $\zeta$).  For $s<t$, assume
that $\psi^i(s)$ is built, and let $-v$ be the largest time realization of
$\o^V_{\psi^i(s)}[-t,-s[$, and let $(j,\zeta)$ its associated mark.  Then,
\[
\forall u\in [s,v[\quad \psi^i(u)=\psi^i(s),\quad\text{and}\quad
\psi^i(v)=\psi^i(s)\cup\{j\}.
\]
Note that for any $t>0$, $\psi^i(t)$ is $\sigma(\o^V[-t,0[)$-measurable,
and that for any subset of labels $a\subset\{1,\dots,N\}$ containing
$i$, we have
\begin{equation}\label{step-1}
\acc{\psi^i(t)=a}\in \sigma(\o^V_a[-t,0[),\quad\text{and}\quad
\acc{\psi^i(t)=a,\ \xi_t(i)=x}\in \sigma(\o_a[-t,0[).
\end{equation}
The next lemma says that the sets of labels associated to two different
particles intersect with probability of order
$1/N$. 
\begin{lema}\label{lem-rate}
For $i,j$ distinct labels, and $t>0$
\begin{equation}\label{step-key}
\P\pare{\psi^i(t)\cap \psi^j(t)\not= \emptyset}\le
\frac{1}{N-1}\pare{e^{2{C}t}-1}.
\end{equation}

\end{lema}
\begin{proof}
First, we show that the rate of growth of $\psi^i(t)$ is at 
most exponential. 
It is clear from the construction of $\psi^i$, that its rate
of growth at time $t$ is at most $C \psi^i(t)$, and that it grows
by adding one label (from $\{1,\dots,N\}\setminus \psi^i(t)$)
uniformly at random. Thus,
\begin{equation}\label{step-3}
\frac{d \E\cro{|\psi^i(t)|\big|\sigma(\o[-t,0[)}}{dt}\le
C|\psi^i(t)|\Longrightarrow
\E[|\psi^i(t)|]\le \exp\pare{{C}t}.
\end{equation}
Second, we show that for $i,j$ two distinct labels, 
\begin{equation}\label{step-2}
\P\pare{\psi^i(t)\cap \psi^j(t)\not= \emptyset}\le
\frac{2C}{N-1}\int_0^t \E[\psi^i(s)]\E[\psi^j(s)]ds.
\end{equation}
Note that
\begin{equation}\label{step-4}
\P\pare{\psi^i(t)\cap \psi^j(t)\not=\emptyset}=
\int_0^t\E\cro{\frac{d\P\pare{\psi^i(s)\cap \psi^j(s)\not=\emptyset
\big|\sigma(\o[-s,0[)}}{ds}}ds
\end{equation}
and
\begin{equation}\label{step-5}
\begin{split}
\E\cro{\frac{d\P\pare{\psi^i(s)\cap \psi^j(s)\not=\emptyset
\big|\sigma(\o[-s,0[)}}{ds}}=& \frac{2C }{N-1} 
\E\cro{\ind_{\psi^i(s)\cap \psi^j(s)= \emptyset} |\psi^i(s)\|\psi^j(s)|}\\
=&\frac{2C }{N-1}\sum_{a\cap b=\emptyset} |a\|b|\
\P\pare{ \psi^i(s)=a,\psi^j(s)=b}\\
=&\frac{2C }{N-1}\sum_{a\cap b=\emptyset} |a\|b|\
\P\pare{ \psi^i(s)=a}\P\pare{\psi^j(s)=b}\\
\le &\frac{2C }{N-1}\E\cro{ \psi^i(s)}\E\cro{\psi^j(s)}.
\end{split}
\end{equation}
We used in \reff{step-5} that
for two non-overlapping subsets of labels $a$ and $b$,
$\{\psi^i(t)=a\}$ and $\{\psi^j(t)=b\}$ are independent by \reff{step-1}.

This concludes the proof of \reff{step-2}. Now, \reff{step-key} follows
from \reff{step-2} and \reff{step-3}.
\end{proof}

\section{Proof of Proposition~\ref{chaos-pablo}}
\label{sec.correlations}
We need to show that for any $x,y\in \Lambda$, any time $t\ge 0$,
and initial configuration $\xi$,
\be{step-11}
\big|\E[\eta(\xi_t)(x)\,\eta(\xi_t)(y)]-
\E[\eta(\xi_t)(x)]\E[\eta(\xi_t)(y)]\big|\;\le\; 2Ne^{2{C}t}.
\ee
Here and throughout this section we use $\E$ and $\P$ to denote $\E_\xi^N$ and $\P_\xi^N$ respectively. Using \reff{a31}, the difference of expectations in the left hand side of
\eqref{step-11} is
\be{step-13}
\sum_{i\le N}\sum_{j\le N}
\big[\P(\xi_t(i)=x,\xi_t(j)=y)-\P(\xi_t(i)=x)\,\P(\xi_t(j)=y)\big].
\ee
For a subset $a$, $\{\xi_t(i)=x, \psi^i(t)=a\}$ 
is $\sigma(\o_a[-t,0[)$-measurable, by remark \reff{step-1}.
Thus, for two non-overlapping subsets of labels $a$ and $b$
\begin{align*}
\P\big(\psi^i(t)=a,\,&\psi^j(t)=b,\, \xi_t(i)=x,\xi_t(j)=y\big)\\
&=\;\P\pare{\psi^i(t)=a, \xi_t(i)=x}\P\pare{\psi^j(t)=b, \xi_t(j)=y}.
\end{align*}
Compute a generic term on the right hand side of \reff{step-13} with $i\not= j$:
\be{step-14}
\begin{split}
\P(\{\xi_t(i)=x,\xi_t(j)=y\})&=\;\P(\psi^i(t)\cap \psi^j(t)\not=
\emptyset, \xi_t(i)=x,\xi_t(j)=y)\\
&\qquad+\sum_{a\cap b=\emptyset}
\P\pare{\psi^i(t)=a,\psi^j(t)=b, \xi_t(i)=x,\xi_t(j)=y}\\
&=\;\P(\psi^i(t)\cap \psi^j(t)\not=\emptyset, \xi_t(i)=x,\xi_t(j)=y)\\
&\qquad+\sum_{a\cap b=\emptyset}
\P\pare{\psi^i(t)=a,\xi_t(i)=x}
\P\pare{\psi^j(t)=b, \xi_t(j)=y}.
\end{split}
\ee 
To compute $\P(\xi_t(i)=x)\P(\xi_t(j)=y)$, we can think of two
independent marked-point processes driving the evolution (put a tilda for the
independent copy), and we have a decomposition similar to \reff{step-14} for
$i\neq j$:
\be{step-15}
\begin{split}
\P(\xi_t(i)=x)\,\P(\xi_t(j)=y)&=\;\P(\psi^i(t)\cap \tilde
\psi^j(t)\not=\emptyset, \xi_t(i)=x,\tilde \xi_t(j)=y)\\
&\qquad+\sum_{a\cap b=\emptyset}
\P\pare{\psi^i(t)=a,\tilde \psi^j(t)=b, \xi_t(i)=x,\tilde \xi_t(j)=y}\\
&=\;\P(\psi^i(t)\cap\tilde \psi^j(t)\not=\emptyset, \xi_t(i)=x,
\tilde\xi_t(j)=y)\\
&\qquad+\sum_{a\cap b=\emptyset}
\P\pare{\psi^i(t)=a,\xi_t(i)=x}\P\pare{\psi^j(t)=b, \xi_t(j)=y}.
\end{split}
\ee
Subtracting \eqref{step-14} and \eqref{step-15} we get that for $i\neq j$,
\begin{align*}
|\P(\{\xi_t(i)=&x,\xi_t(j)=y\})-
\P(\{\xi_t(i)=x)\P(\xi_t(j)=y\})| \\
&\le \;\P(\psi^i(t)\cap \psi^j(t)\not=
\emptyset)+\P(\psi^i(t)\cap \tilde \psi^j(t)\not=\emptyset)
\;\le\; \frac{2}{N-1}(e^{2{C}t}-1),
\end{align*}
by Lemma~\ref{lem-rate} (we have used that the Lemma holds also for $\psi^i(t)\cap\tilde \psi^j(t)$).  Thus, by summing over $i$ and $j\in \{1,\dots,N\}$,
and noting that there are $N$ diagonal terms which bring a factor $N$ when
$i=j$, we get the desired bound.
\section{Proof of Theorem~\ref{main.theorem}.}
\label{sec.conv.means}
In view of \reff{triangle0}, we first estimate 
$\Exi\cro{\|m(\xi_t)-T_tm(\xi))\|}$. It is more convenient
to work with $l^2$-norm, rather than total variation norm.
For a function $\v:\Lambda\to\R$, we denote its $l^2$-norm
as
\[
\|\v\|_2=\Bigl(\sum_{x\in \Lambda} \v^2(x)\Bigr)^{\frac{1}{2}}.
\]
By Cauchy-Schwarz, note that if $\mu,\nu$ are probabilities on $\Lambda$,
\[
\|\mu-\nu\|_2\le \|\mu-\nu\|\le \sqrt{|\Lambda}|\ \|\mu-\nu\|_2.
\]
To estimate $\Exi\cro{\|m(\xi_t)-T_tm(\xi)\|_2}$  note that
\be{step-16}
\Exi\cro{\|m(\xi_t)-T_tm(\xi)\|_2}\;\le\;
\Exi\cro{\|m(\xi_t)-\Exi\cro{m(\xi_t)}\|_2}\,+\,
\|\Exi\cro{m(\xi_t)}-T_tm(\xi)\|_2.
\ee
Taking $y=x$ in \eqref{correlations} we obtain
\be{var-N}
\E\cro{\pare{m_x(\xi_t)-\E[m_x(\xi_t)]}^2}\,\le\, \frac{2e^{2{C}t}}{N}.
\ee
By \reff{var-N} and Jensen's inequality, we have
\be{step-17}
\cro{\Exi\|m(\xi_t)-\Exi\cro{m(\xi_t)}\|_2}^2\le
\Exi\cro{\|m(\xi_t)-\Exi\cro{m(\xi_t)}\|_2^2}\le \frac{2|\Lambda|
e^{2C t}}{N}.
\ee
The second term in \reff{step-16} is dealt with in the following lemma.
\begin{lema}\label{teo.conv.means}
For any $T>0$,
\begin{eqnarray}
  \lim_{N\to\infty}\max_{0\le t \le T}\max_{\xi \in \Lambda^N} 
\|\Exi\cro{m(\xi_t)}-T_tm(\xi)\|_2 = 0\,. \label{lgn2}
\end{eqnarray}
\end{lema}

\begin{proof}
We introduce some simplifying notations.
\be{symbol-1}
u_x(t)=\Exi\cro{m_x(\xi_t)},\quad
\text{and}\quad v_x(t)=T_tm(\xi)(x).
\ee
We  show that there is a constant $B$ such that for any $t>0$
\be{step-30}
\frac{d}{dt}\|u(t)-v(t)\|_2^2\le B \|u(t)-v(t)\|_2^2+\frac{4
e^{2{C}t}\sum_y q(y,0)}{N}.
\ee
Since $\|u(0)-v(0)\|_2=0$, the result follows at once by means of Gronwall's inequality.

We fix $t>0$, and we often omit to display the time-dependence.
From \reff{generator.eta} 
\be{a82} 
\frac{d u_x}{dt} =\sum_{y \in \Lambda} q(y,x)u_y + \sum_{y\in \Lambda} q(y,0) 
u_xu_y+R_x(\xi,t)
\ee
$$
\text{where}\quad R_x(\xi,t)=\sum_{y \in\Lambda}q(y,0)
\cro{\frac{N}{N-1}\Exi\cro{m_y(\xi_t)m_x(\xi_t)}
-\Exi\cro{m_y(\xi_t)}\Exi\cro{m_x(\xi_t)}}.
$$
Proposition~\ref{chaos-pablo}, implies that
\be{step-31}
\sup_{x\in \Lambda}\sup_{\xi} |R_x(\xi,t)|\le \frac{ 2e^{2Ct}\sum_y q(y,0)}
{N}.
\ee
On the other hand, from \eqref{kolmogorov.cond.evolution} 
\be{step-32}
\frac{d v_x}{dt}
=\sum_{y \in \Lambda} q(y,x)v_y + \sum_{y\in \Lambda} q(y,0) v_xv_y.
\ee
Subtracting \reff{step-32} from \reff{a82}, we obtain
\[
\frac{d (u_x-v_x)}{dt}=\sum_{y \in \Lambda} q(y,x)(u_y-v_y)+
\sum_{y\in \Lambda} q(y,0) (u_xu_y-v_xv_y)+\ R_x(\xi,t).
\]
Now, 
\be{step-33}
\begin{split}
\frac{d}{dt}\frac{1}{2}\|u(t)-&v(t)\|_2^2=\sum_x
\frac{d (u_x-v_x)}{dt}(u_x-v_x)\\
=&\sum_{x\in\Lambda}\sum_{y\in \Lambda} q(x,y)(u_y-v_y)(u_x-v_x) +\sum_{x\in \Lambda}\sum_{y\in\Lambda} q(y,0) (u_xu_y-v_xv_y)(u_x-v_x)\\
&\ +\sum_x R_x (u_x-v_x).
\end{split}
\ee

We deal with each term on the right hand side of \reff{step-33}
in turns. First,
\be{first_term}
|\sum_{x\in\Lambda}\sum_{y\in \Lambda} q(x,y)(u_y-v_y)(u_x-v_x)|
\le \pare{\sum_{x,y\in \Lambda} q^2(x,y)}^{1/2} \|u-v\|_2^2.
\ee
To deal with the second term observe that $u_xu_y-v_xv_y = v_x(u_y-v_y) + u_y(u_x-v_x)$ and hence the second term in \eqref{step-33} term equals
\be{second_term}
\sum_{x\in \Lambda}\sum_{y\in\Lambda} q(y,0)v_x (u_y-v_y)(u_x-v_x) + \sum_{x\in \Lambda}\sum_{y\in\Lambda} q(y,0)u_y (u_x-v_x)(u_x-v_x).
\ee
Since $v$ is bounded, 
the first term in \eqref{second_term} can be treated as \eqref{first_term} with $q(x,y)$ replaced by $\tilde q(x,y) := q(y,0)v_x$. 
The second term in \eqref{second_term} equals
$$
\|u -v\|_2^2\sum_y q(y,0) u_y \le  \|u -v\|_2^2\sup_y q(y,0). 
$$
Finally, taking into account that $\sum_x v_x=\sum_x u_x=1$, the last term of \eqref{step-33} can be bounded by
$$
2\sup_x|R_x| \le \frac{4e^{2Ct}\sum_y q(y,0)}{N}
$$

Collecting all these computations, we 
obtain \reff{step-30}.
\end{proof}

%
%
%
%

\noindent{\bf Acknowledgements}. This work was initiated during
the semester {\it Interacting Particle systems} at IHP, Paris.
We would like to thank the staff and the organizers for a wonderful
work atmosphere.
A.A. acknowledges the support of the French Ministry of Education 
through the ANR BLAN07-2184264 grant. P.A.F was partially supported by FAPESP. P.G. is partially supported by Universidad de Buenos Aires under grang X447, by ANPCYT PICT 2006-1309 and CONICET PIP 5478/1438.

\end{document}